\documentclass[12pt]{amsart}
\usepackage{amsmath, amssymb, mathabx, enumitem, amsfonts, mathrsfs}
\usepackage{color}
\usepackage[colorlinks=true,linkcolor=blue,urlcolor=blue,citecolor=blue, pagebackref]{hyperref}
\usepackage[alphabetic]{amsrefs}
\usepackage{ulem}
\usepackage{tikz}
\usepackage{tikz-cd}

\usepackage[margin=0.7in]{geometry}

\newcommand{\C}{\mathbb{C}}

\newcommand{\Z}{\mathbb{Z}}

\newcommand{\mf[1]}{\mathfrak{#1}}

\newtheorem{theorem}{Theorem}[section]

\newtheorem{remark}[theorem]{Remark}

\newtheorem{corollary}[theorem]{Corollary}
\newtheorem{question}[theorem]{Question}
\newtheorem{proposition}[theorem]{Proposition}

\newtheorem{lemma}[theorem]{Lemma}

\newtheorem{definition}[theorem]{Definition}
\newtheorem{definition/lemma}[theorem]{Definition/Lemma}

\DeclareMathOperator{\ad}{ad} 
\DeclareMathOperator{\Ad}{Ad}
\DeclareMathOperator{\Exp}{Exp}
\DeclareMathOperator{\End}{End}

\DeclareMathOperator{\Aut}{Aut}

\title[Jacobson--Morozov]{Toward a Jacobson--Morozov theorem for Kac--Moody Lie algebras}

\author{Sam Jeralds}
\address{Department of Mathematics, University of North Carolina, Chapel Hill, NC 27599}
\email{sjj280@live.unc.edu}

\begin{document}
\maketitle

\begin{abstract} 
For a finite-dimensional semisimple Lie algebra $\mf[g]$, the Jacobson--Morozov theorem gives a construction of subalgebras $\mf[sl]_2 \subset \mf[g]$ corresponding to nilpotent elements of $\mf[g]$. In this note, we propose an extension of the Jacobson--Morozov theorem to the symmetrizable Kac--Moody setting and give a proof of this generalization in the case of rank two hyperbolic Kac--Moody algebras. We also give a proof for an arbitrary symmetrizable Kac--Moody algebra under some stronger restrictions.  
\end{abstract}

%Intro
\section{Introduction}

Let $\mf[g]$ be a finite-dimensional, semisimple Lie algebra over $\C$. In this note, we are concerned about embeddings $\mf[sl]_2 \hookrightarrow \mf[g]$. In particular, we want to consider not just the embeddings but also the image of the ``nil-positive" element of $\mf[sl]_2$, as given by the following definition. 

\begin{definition} For any element $e \in \mf[g]$, we say that $e$ can be \emph{extended into an $\mf[sl]_2$-triple} if there is a Lie algebra homomorphism $\theta_e: \mf[sl]_2(\C) \to \mf[g]$ such that 
$$
\theta_e \left( \begin{pmatrix} 0 & 1 \\ 0  & 0 \end{pmatrix} \right) = e.
$$
\end{definition}

Prototypical examples of elements which can be extended into $\mf[sl]_2$-triples are the root vectors $e_\beta \in \mf[g]_\beta$, where $\mf[g]_\beta$ is the $\beta$-root space of $\mf[g]$ (see, for example, \cite{Hum}). These root vectors also have the common property of being nilpotent elements of $\mf[g]$; that is, the associated adjoint operator $\ad_{e_\beta} \in \End(\mf[g])$ is nilpotent. 

A classical result in Lie theory, known as the Jacobson--Morozov theorem, says that these two properties are linked for arbitrary elements in $\mf[g]$: 

\begin{theorem}[Jacobson--Morozov]

Let $e \in \mf[g]$ be a nilpotent element. Then $e$ can be extended into an $\mf[sl]_2$-triple. 
\end{theorem}

Since its discovery, the Jacobson--Morozov theorem has been adapted to far-reaching generalizations beyond characteristic zero (as in \cite{Prem}) and semisimple Lie algebras (as in \cite{AK} and \cite{OSul}). The Jacobson--Morozov theorem has, over the last six decades, played a vital role in the study of the structure theory of semisimple Lie algebras, representation theory, and kickstarted through the work of Kostant and many others the study of the geometry of nilpotent orbits and geometric representation theory; for some historical context and in-depth treatment of these, we refer to \cite{CM} and \cite{CG}. 

Now, let $\mf[g]$ be a symmetrizable Kac--Moody algebra, a (typically infinite-dimensional) generalization of a semisimple Lie algebra. Our goal is to develop an analogue of the Jacobson--Morozov theorem for suitable elements $e \in \mf[g]$ that recovers the classical result when $\mf[g]$ is finite-dimensional. To do so, we must overcome the technical difficulties that accompany the infinite-dimensional behavior of $\mf[g]$, including determining a suitable analogue of nilpotent elements. The main results of this note are given in the following theorem (cf. Theorems \ref{rank 2 theorem}, \ref{regdomthm} in the text); see Section \ref{S2} and Question \ref{Q1} for notation, conventions, and  context. 

\begin{theorem} Let $\mf[g]$ be a symmetrizable Kac--Moody algebra, $w \in W$ an element of the Weyl group $W$, and $d \in \Z_{\geq 1}$. Then any element $e \in \mf[g]_w \cap \mf[g]_d$ can be extended into an $\mf[sl]_2$-triple where 
\begin{itemize}
\item[(a)] $\mf[g]$ is a rank two hyperbolic Kac--Moody algebra, or 
\item[(b)] $\mf[g]$ is any symmetrizable Kac--Moody algebra and there is a regular grading on $\mf[g]$. 
\end{itemize}
\end{theorem}

The proofs of these results, occupying Sections \ref{S3} and \ref{S4}, are primarily combinatorial in nature and follow from a careful study of the real roots of $\mf[g]$. In particular, we make extensive use of the complete knowledge of the real roots of rank two Kac--Moody algebras, and in the general symmetrizable case the notion of \emph{regular subalgebras} or \emph{$\pi$-systems}. Regular subalgebras have a long history in the study of semisimple Lie algebras, beginning with Dynkin \cite{Dyn}, and have appeared alongside the Jacobson--Morozov theorem in the study of $\mf[sl]_2$ subalgebras. We present here, to our knowledge, the first explicit use of these $\pi$-systems to construct $\mf[sl]_2$-triples in the general Kac--Moody setting. 

%In Section \ref{S5}, following an argument outlined to us by Shrawan Kumar \cite{Ku}, we adopt much of the proof of the Jacobson--Morozov theorem as given by Kostant \cite{Kos} into our context. In particular, we give a sufficient condition which would complete the proof of our analogue in the general symmetrizable case, as given in the following theorem (cf. Theorem \ref{maintheorem}). 
%
%\begin{theorem} Let $\mf[g]$ be a symmetrizable Kac--Moody algebra, $w \in W$ an element of the Weyl group, and $d \in \Z_{\geq 1}$. Let $e \in \mf[g]_w \cap \mf[g]_d$ and suppose that for any $y \in \mf[g]$, 
%$$
%[e,[e,y]] \implies (e |y)=0
%$$
%for the invariant bilinear form $(\cdot | \cdot)$ on $\mf[g]$. Then $e$ can be extended into an $\mf[sl]_2$-triple. 
%\end{theorem}

%Finally, in Section \ref{S6} we give a brief discussion about $\mf[sl]_2$-triples, finite-dimensional semisimple subalgebras of $\mf[g]$, and conjugation in $\mf[g]$. 

\subsection*{Acknowledgements}
We are deeply indebted to Shrawan Kumar, who first suggested this project and offered patient guidance throughout.

\section{Kac--Moody algebras, gradings, and locally nilpotent elements} \label{S2}

\subsection{Notation and preliminaries}
Throughout, $\mf[g]=\mf[g](A)$ will denote a symmetrizable Kac--Moody Lie algebra with generalized Cartan matrix $A$. Because $\mf[g]$ is symmetrizable, we have an invariant, symmetric bilinear form $(\cdot | \cdot)$ on $\mf[g]$. We denote by $\mf[h] \subset \mf[b] \subset \mf[g]$ a choice of Cartan subalgebra and (positive) Borel subalgebra of $\mf[g]$, respectively. We fix a root system $\Phi$ of $\mf[g]$ with positive (resp., negative) roots $\Phi^+$ (resp., $\Phi^-$) consistent with the choice of Cartan and Borel subalgebras. The simple roots are denoted by $\{\alpha_i\} \subset \Phi$. We also write $\Phi = \Phi_{Re} \sqcup \Phi_{Im}$, the disjoint union of the roots into the real and imaginary roots, respectively. With this, we have the root space decomposition of $\mf[g]$ under the adjoint action of $\mf[h]$ given by 
$$
\mf[g]=\mf[h] \oplus \bigoplus_{\beta \in \Phi} \mf[g]_\beta.
$$
Finally, we denote by $W$ the Weyl group of $\mf[g]$, with generators the simple reflections $\{s_i\}$, which act on simple roots via $s_i(\alpha_j)=\alpha_j-\alpha_j(\alpha_i^\vee)\alpha_i=\alpha_j-A_{ij}\alpha_i$.

\subsection{Gradings and locally nilpotent elements of Kac--Moody algebras}

As the statement of the Jacobson--Morozov theorem for semisimple Lie algebras is concerned with nilpotent elements, a first step in generalizing to the symmetrizable setting is finding a suitable analogue of this notion. By the adjoint action, we can still consider the linear operator $\ad_x \in \End(\mf[g])$; however, since $\mf[g]$ is generically infinite-dimensional, considering those $x \in \mf[g]$ for which $\ad_x$ is nilpotent is too restrictive. Instead, we relax this condition with the following definition. 

\begin{definition} An element $x \in \mf[g]$ is called \emph{locally nilpotent} if, for any $y \in \mf[g]$, there exists some $N_y \in \Z_{\geq 0}$ such that $\ad_x^{N_y}(y)=0$. 
\end{definition}

Note that when $\mf[g]$ is finite-dimensional, any locally nilpotent element is nilpotent. In general, prototypical examples of locally nilpotent elements are homogeneous root vectors $x_\beta \in \mf[g]_\beta$ where $\beta \in \Phi_{Re}$. In order to set up a more general construction of locally nilpotent elements in which we will be interested, we first introduce the following combinatorial definition. 

\begin{definition} For a symmetrizable Kac--Moody algebra $\mf[g]$ with Weyl group $W$, define for an element $w \in W$ the set 
$$
\Phi_w:=\{\beta \in \Phi^+: w^{-1}(\beta) \in \Phi^-\}.
$$
\end{definition}

That is, $\Phi_w$ is the set of positive roots that are ``inverted" to negative roots by $w^{-1}$. These sets, appropriately termed \emph{inversion sets}, appear throughout combinatorics and have applications to representation theory and the geometry of flag varieties. One particular well-known property that follows directly from the definition is that $\Phi_w$ is a  \textit{closed set} of roots; that is, if $\alpha$, $\beta \in \Phi_w$ and $\alpha+\beta \in \Phi$, then necessarily $\alpha+\beta \in \Phi_w$. It therefore makes sense to define a subalgebra $\mf[g]_w \subset \mf[g]$ via 
$$
\mf[g]_w:= \bigoplus_{\beta \in \Phi_w} \mf[g]_\beta.
$$
We will from now on restrict our attention to elements $e \in \mf[g]_w$ when considering an analogue of the Jacobson--Morozov theorem, via the following result taken from \cite{Ku2}*{Theorem 10.2.5}. 

\begin{theorem} Fix an element $e \in \mf[g]_w$ with $\mf[g]_w$ as above. Then $e$ is locally nilpotent in $\mf[g]$. 
\end{theorem}

Next, we introduce a grading on $\mf[g]$ and restrict our attention to those locally nilpotent elements lying in a single graded component; this mimics the finite-dimensional semisimple setting, where any nilpotent element is in a single induced graded piece for some grading (see, for example, \cite{Prem} or \cite{Pom}). To this end, let $G^{min}$ be the minimal Kac--Moody group associated to $\mf[g]$, and let $T \subset G^{min}$ be the maximal torus (we refer to \cite{Ku2} for the full construction and properties of $G^{min}$). Fix $\tau \subset T$ an integral one-parameter subgroup such that $\dot{\tau}$ is dominant; that is, $\alpha_i(\dot{\tau}) \geq 0$ for all simple roots $\alpha_i$. Further, we restrict to $\tau$ such that each eigenspace $\mf[g]_n$ under the action of $\dot{\tau}$ is finite-dimensional. Then we have a graded decomposition into finite-dimensional subspaces 
$$
\mf[g]= \bigoplus_{n \in \Z} \mf[g]_n.
$$
With this in mind, we are now interested in locally nilpotent elements $e \in \mf[g]_w$ such that $e \in \mf[g]_d$ for some $d \geq 1$. Then the version of the Jacobson--Morozov theorem which we would like to consider is as follows: 

\begin{question} \label{Q1} Let $\mf[g]$ be a symmetrizable Kac--Moody algebra, $w \in W$ a Weyl group element, and $\tau \subset T$ a dominant one-parameter subgroup as above. Fix $d \geq 1$, and consider an element 
$$
e \in \mf[g]_w \cap \mf[g]_d.
$$
Can $e$ be extended into an $\mf[sl]_2$-triple?
\end{question}

When $\mf[g]$ is an affine Kac--Moody Lie algebra, Ressayre \cite{Res} gave an affirmative answer to Question \ref{Q1}. However, his proof utilized the loop algebra construction for affine Lie algebras, which has no analogue in the general setting. In the following two sections, we first give an affirmative answer to Question \ref{Q1} completely for the rank two hyperbolic Kac--Moody algebras. Then, we give partial results for general symmetrizable $\mf[g]$ by imposing a stronger restriction on the one-parameter subgroup $\tau$. We also introduce interesting new machinery from the study of regular subalgebras that is applicable for our construction of relevant $\mf[sl]_2$-triples. Throughout, by writing $e \in \mf[g]_w \cap \mf[g]_d$ we work with suitably generic $w \in W$, one-parameter subgroups $\tau$, and positive integers $d \geq 1$ as in Question \ref{Q1}, unless otherwise noted.

\section{Rank two hyperbolic Kac--Moody algebras} \label{S3}

In this section, we consider the case where $\mf[g]$ is a rank two hyperbolic Kac--Moody algebra. These are given by generalized Cartan matrices of the form
$$
A(a,b):= \begin{pmatrix} 2 & -b \\ -a & 2 \end{pmatrix}
$$
where $a, b \in \Z_{\geq 1}$ such that $ab \geq 5$. We will denote by $\mf[g]=\mathcal{H}(a,b)$ the associated algebra. In the case when $a=b$, we denote this algebra by $\mathcal{H}(a)$. As the algebras $\mathcal{H}(a,b)$ are the simplest symmetrizable Kac--Moody algebras beyond the semisimple and affine setting, a lot is known about their root systems and internal structure (see \cite{CKMS} and the references therein). We will first consider the case of $\mathcal{H}(a)$, which is more straightforward, to highlight the key idea of the general approach. 

\subsection{$\mathcal{H}(a)$ and $\mf[sl]_2$-triples}

First, recall that the Weyl group of $\mathcal{H}(a)$ (and in fact for any $\mathcal{H}(a,b))$ is the infinite dihedral group $W \cong \langle s_1, s_2 | s_1^2=s_2^2=1 \rangle$. From this, we can derive the following lemma, a textbook exercise (\cite{Kac}, Exercise 5.25).

\begin{lemma} \label{symroots} The set of positive real roots for $\mathcal{H}(a)$ are given by $\Phi^+_{Re} = \{ b_n \alpha_1+ b_{n-1} \alpha_2, \ b_{n-1} \alpha_1 + b_{n} \alpha_2 \}$ where the sequence $\{b_n\}$ is defined recursively via $b_0=0$, $b_1=1$, $b_n=ab_{n-1}-b_{n-2}$.
\end{lemma}

These real roots correspond to the integral points on the two branches of the hyperbola $x^2-axy+y^2=1$. The sequence $\{b_n\}$ (and similar sequences for real root coefficients in the nonsymmetric case $\mathcal{H}(a,b)$) is a ``generalized Fibonacci sequence," as first explored by Feingold \cite{Fei}. In particular, the following is a consequence of the recursive definition for $\{b_n\}$ along with the fact that $a \geq 3$.

\begin{lemma} \label{symmonotone} The sequence $\{b_n\}$ is strictly increasing: $0=b_0 < b_1 < b_2 < \cdots$ 
\end{lemma}

\begin{proof} We induct on $n$ to show that $b_n-b_{n-1} >0$ for all $n \geq 1$. For the base case $n=1$ this is clear, as $b_1-b_0=1-0=1>0$. Assuming $b_n-b_{n-1} >0$, we have via the recursive definition that 
$$
b_{n+1}-b_n=ab_n-b_{n-1}-b_n=(a-2)b_n+(b_n-b_{n-1}) >0
$$
as $a-2 \geq 1$, $b_n \geq 0$, and the inductive hypothesis gives $b_n-b_{n-1} >0$. 
\end{proof}

We can now prove the following proposition, from which the Jacobson--Morozov theorem will follow immediately. 

\begin{proposition} \label{symprop} The space $\mathcal{H}(a)_w \cap \mathcal{H}(a)_d$, if nonempty, corresponds to a single real root space. 
\end{proposition}

\begin{proof}
As $W$ is the infinite dihedral group and the Cartan matrix is symmetric, without loss of generality assume $l(w)=n$ and $w=s_1 s_2 s_1 \cdots$. Then we have 
$$
\Phi_w = \{ b_k \alpha_1 + b_{k-1} \alpha_2 \}_{k=1}^n.
$$

Suppose that $\alpha, \beta \in \Phi_w$ such that $\alpha(\dot{\tau})=\beta(\dot{\tau})=d \geq 1$. Write $\alpha = b_i \alpha_1 + b_{i-1} \alpha_2$, $\beta= b_j \alpha_1 + b_{j-1} \alpha_2$, and suppose that $i >j$. Then we have 
$$
0 = \alpha(\dot{\tau})-\beta(\dot{\tau}) = (b_i - b_j )\alpha_1(\dot{\tau}) + (b_{i-1} - b_{j-1}) \alpha_2(\dot{\tau}).
$$

But as $\tau$ is dominant, and as the sequence $\{b_k\}$ is strictly increasing, we have necessarily that $\alpha_1(\dot{\tau})=\alpha_2(\dot{\tau}) =0$, contradicting $d \geq 1$. Thus, there cannot be two distinct roots in $\Phi_w$ evaluating to $d$ on $\dot{\tau}$. 
\end{proof}

Therefore, just by root space considerations, if $0 \neq e \in \mathcal{H}(a)_w \cap \mathcal{H}(a)_d$, necessarily $e=e_\beta$ for some real root $\beta$, so we have the following corollary. 

\begin{corollary} \label{symcor} For $\mf[g]=\mathcal{H}(a)$, if $e \in \mf[g]_w \cap \mf[g]_d$, then $e$ can be extended into an $\mf[sl]_2$-triple. 
\end{corollary}

\begin{remark}
In a recent work, Tsurusaki \cite{Tsu} also constructs $\mf[sl]_2$-triples for vectors in the linear span of the real root spaces for symmetric rank two Kac--Moody algebras. Interestingly, due to differences in the initial set up of the approaches therein and in this present work, there is no overlap in the elements of $\mathcal{H}(a)$ extended into $\mf[sl]_2$-triples. 
\end{remark}

\subsection{The general case $\mathcal{H}(a,b)$}

We next consider the case for $\mathcal{H}(a,b)$, where $a \neq b$. Up to symmetry, suppose that $a >b$. This case is very similar to the  symmetric case above, but is necessarily more delicate. Again by \cite{Kac}*{Exercise 5.25}, we can compute the coefficients of the real roots of $\mathcal{H}(a,b)$ as recursive sequences. We refer instead to the paper of Carbone et. al. \cite{CKMS} (and use notation as given therein) which goes into great detail about these coefficients. In particular, we partition the positive real roots of $\mathcal{H}(a,b)$ into the following:
$$
\begin{aligned}
&\alpha_j^{LL}:=(s_1s_2)^j \alpha_1, \hspace{5em}  \alpha_j^{LU}:=(s_2s_1)^js_2 \alpha_1 \\
&\alpha_j^{SU}:=(s_2s_1)^j \alpha_2, \hspace{5em} \alpha_j^{SL}:=(s_1s_2)^j s_1 \alpha_2.
\end{aligned}
$$ 

We then have, using this notation, the following lemma taken from \cite{CKMS}*{Lemma 2.1} which gives recursive definitions of the root coefficients, generalizing Lemma \ref{symroots}.

\begin{lemma} \label{CKMScoeff} For all nonnegative integers $j$, 

$$
\begin{aligned}
&\alpha_j^{LL}=\eta_j \alpha_1 + a \gamma_j \alpha_2, \hspace{5em}  \alpha_j^{LU}= \eta_j \alpha_1 + a \gamma_{j+1} \alpha_2, \\
&\alpha_j^{SU}= b \gamma_j \alpha_1 + \eta_j \alpha_2, \hspace{5em} \alpha_j^{SL}= b \gamma_{j+1} \alpha_1 +\eta_j \alpha_2,
\end{aligned}
$$ 
where 
\begin{itemize}
\item[(a)] $\gamma_0=0$, $\gamma_1=1$, $\eta_0=1$, $\eta_1=ab-1$;
\item[(b)] $\eta_j= ab\gamma_j - \eta_{j-1}$;
\item[(c)] $\gamma_j = \eta_{j-1}-\gamma_{j-1}$;
\item[(d)] Both sequences $X_j=\gamma_j$ and $X_j=\eta_j$ satisfy the recurrence relation 
$$
X_j = (ab-2)X_{j-1}-X_{j-1}.
$$
\end{itemize}
\end{lemma}

\begin{remark} Note that, by the same proof as Lemma \ref{symmonotone}, Lemma \ref{CKMScoeff}(d) ensures that the sequences $\{\gamma_j\}$ and $\{\eta_j\}$ are strictly increasing. 
\end{remark}

Now, we consider the sets $\Phi_w$ for $w \in W$. Again, as $W$ is the infinite dihedral group, reduced words for $w$ are given initially by either $w_1=s_1s_2s_1 \cdots $ or $w_2=s_2s_1s_2 \cdots$. Then we get for these initial sequences that 
$$
\begin{aligned}
\Phi_{w_1} &\subset \{\alpha_j^{LL}, \alpha_j^{SL} \}_{j \geq 0}, \\
\Phi_{w_2} &\subset \{\alpha_j^{SU}, \alpha_j^{LU} \}_{j \geq 0}.
\end{aligned}
$$
To replicate the proof Proposition \ref{symprop}, we need to consider when two roots $\alpha, \beta \in \Phi_w$ can satisfy $\alpha(\dot{\tau})=\beta(\dot{\tau})=d$ for some $d \geq 1$. The key fact in the proof was the strict monotonicity of the coefficients of the roots in $\Phi_w$. In the nonsymmetric case, we make use of the following similar inequalities of \cite{CKMS}*{Lemmas 3.1 and 3.5}. 

\begin{lemma} \label{inequalities}

\begin{enumerate}
\item[(a)]  If $a >b>1$, then 
$$
\begin{aligned} 
0&=b\gamma_0 < \eta_0 < b \gamma_1 < \eta_1 < b \gamma_2 < \cdots ,\\
0&= a \gamma_0 < \eta_0 < a \gamma_1 < \eta_1 < a \gamma_2 < \cdots .
\end{aligned}
$$

\item[(b)] If $a > b=1$, then 
$$
\begin{aligned}
0&= \gamma_0 < \eta_0 = \gamma_1 < \gamma_2 < \eta_1 < \gamma_3 < \eta_2 \cdots , \\
0&= a \gamma_0 < \eta_0 < \eta_1 < a \gamma_1 < \eta_2 < a \gamma_2 < \eta_3 < a \gamma_3 < \cdots .
\end{aligned}
$$
\end{enumerate}

\end{lemma}

With these inequalities, we can now arrive at the following proposition.

\begin{proposition} \label{permissableroots} 
\begin{enumerate}
\item[(a)] If $a > b > 1$, the space $\mathcal{H}(a,b)_w \cap \mathcal{H}(a,b)_d$ is either empty or corresponds to a single real root space. 

\item[(b)] Else, if $a>b=1$, the space $\mathcal{H}(a,1)_w \cap \mathcal{H}(a,1)_d$ is either empty or corresponds to a single real root space \textbf{except in the cases}

\begin{enumerate}

\item[(i)] $w=s_1 s_2  \dots$, \ $l(w) \geq 2$, \ $\alpha= \alpha_1+\alpha_2$, \ $\beta= \alpha_1$, and $\alpha_1 (\dot{\tau})=d$, $\alpha_2(\dot{\tau})=0$, 

\item[(ii)] $w=s_2s_1s_2 \dots$, \ $l(w) \geq 3$, \ $\alpha=\alpha_1+ a \alpha_2$, \ $\beta= \alpha_1+ (a-1) \alpha_2$, and $\alpha_1(\dot{\tau})=d$, $\alpha_2(\dot{\tau})=0$. \end{enumerate}

\end{enumerate}
\end{proposition}

\begin{proof}
\begin{itemize}

\item[(a)] Since the sequences $\{\eta_j\}$ and $\{\gamma_j\}$ are both strictly increasing, as in Corollary \ref{symcor} we have for any label $J \in \{LL, \ SL, \ LU, \ SU\}$ that $(\alpha_j^{J}-\alpha_k^{J})(\dot{\tau})=0$ forces $j=k$, so that there can be at most one of these roots in $\Phi_w$ with $\alpha_j^J(\dot{\tau})=d \geq 1$. Now, suppose that $\alpha_j^{LL}(\dot{\tau})-\alpha_k^{SL}(\dot{\tau})=0$, respectively $\alpha_j^{LU}(\dot{\tau})-\alpha_k^{SU}(\dot{\tau})=0$. Then we have, respectively, that 
$$
(\eta_j -b\gamma_{k+1}) \alpha_1(\dot{\tau}) + (a \gamma_j - \eta_k) \alpha_2(\dot{\tau})=0
$$
or 
$$
(\eta_j - b\gamma_{k}) \alpha_1(\dot{\tau}) + (a \gamma_{j+1} - \eta_k) \alpha_2(\dot{\tau})=0. 
$$
But by Lemma \ref{inequalities}(a), the coefficient pairs $(\eta_j-b\gamma_{k+1}), (a \gamma_j - \eta_k)$ or $(\eta_j-b\gamma_k), (a \gamma_{j+1} - \eta_k)$ are nonzero and have the same sign. Thus, we cannot have $(\alpha_j^{LL}-\alpha_k^{SL})(\dot{\tau})=0$ or $(\alpha_j^{LU}-\alpha_k^{SU})(\dot{\tau})=0$. Then the result follows in this case.

\item[(b)] Again, the sequences $\{\eta_j\}$ and $\{\gamma_j\}$ are strictly increasing, so that for any label $J$ as above we have $(\alpha_j^J-\alpha_k^J)(\dot{\tau})=0$ forces $j=k$. Now, using that $b=1$ we have 
$$
\begin{aligned}
\alpha_j^{LL}-\alpha_k^{SL} &= (\eta_j - \gamma_{k+1})\alpha_1 + (a \gamma_j - \eta_k) \alpha_2, \\
\alpha_j^{LU}-\alpha_k^{SU} &= (\eta_j-\gamma_k) \alpha_1 + (a \gamma_{j+1}-\eta_k) \alpha_2.
\end{aligned}
$$
Then similarly as above, we use Lemma \ref{inequalities}(b) to see that the coefficient pairs are both nonzero and have the same sign \textbf{except} in the cases 
$$
\begin{aligned}
 j=k=0&: \alpha_0^{SL}-\alpha_0^{LL}=\alpha_2, \\
 j=0, k=1&: \alpha_0^{LU}-\alpha_1^{SU} = \alpha_2.
\end{aligned}
$$
Finally, the pairs $\{ \alpha_0^{SL}, \alpha_0^{LL}\}$ and $\{ \alpha_0^{LU}, \alpha_1^{SU}\}$ can only appear in the sets $\Phi_w$ as described and both satisfy $\alpha_j^{J}(\dot{\tau}) = d \geq 1$ for appropriate labels $j$ and $J$ when $\alpha_2(\dot{\tau})=0$.

\end{itemize}
\end{proof}

So, as in Proposition \ref{symprop} and Corollary \ref{symcor}, we can conclude that the majority of elements $e \in \mathcal{H}(a,1)_w \cap \mathcal{H}(a,1)_d $ can be extended into $\mf[sl]_2$-triples, as we can write $e=e_\beta$ for some real positive root $\beta$. While this does \textbf{not} hold uniformly as in the case of $\mathcal{H}(a)$, the next proposition shows that the remaining cases are off only by an element of $\Aut(\mathcal{H}(a,1))$. For the remainder of this section, we set $\mf[g]=\mathcal{H}(a,1)$ for notational convenience. 

\begin{proposition} \label{final rank two cases} Let $e \in \mf[g]_w \cap \mf[g]_d$ as in Proposition \ref{permissableroots} (b)(i),(ii). Then $e$ can be extended into an $\mf[sl]_2$-triple. 
\end{proposition}

\begin{proof}
First, let $w=s_2s_1s_2 \dots$, $l(w) \geq 3$, and write $e=xe_\beta+ye_{\beta+\alpha_2}$, where $\beta=\alpha_1+(a-1)\alpha_2$, $0 \neq e_\beta \in \mf[g]_\beta$, $e_{\beta+\alpha_2}:= [e_2, e_\beta]$, and $0 \neq x, y \in \C$ (else $e$ is a homogeneous real root vector, in which case the result is immediate). Let $g:=\Exp (e_2) \in G^{min}$ and recall the automorphism coming from the adjoint action $\Ad (g) \in \Aut(\mf[g])$; since $\alpha_2$ is a real root, by \cite{Ku2}*{Exercise 6.2.4} this is indeed an automorphism of $\mf[g]$. But, we can write 
$$
\Ad(g)=\Ad(\Exp(e_2))=\exp(\ad_{e_2})
$$
by \cite{Ku2}*{Exercise 10.2.5}, where $\exp$ is the usual exponential map and as always $\ad_{e_2} \in \End(\mf[g])$ (this is well defined, as $e_2$ is locally nilpotent). With this in mind, we see that 
$$
e=x e_\beta+y e_{\beta+\alpha_2} = \Ad \left( \Exp \left( \frac{y}{x} e_2 \right) \right) (x e_\beta).
$$
But since $xe_\beta$ is a real root vector and can be extended into an $\mf[sl]_2$-triple, by applying the automorphism $\Ad \left( \Exp \left( \frac{y}{x} e_2 \right) \right) $ we construct an $\mf[sl]_2$-triple for $e$, as desired. \\

In the other case, let $w=s_1s_2\dots$, $l(w) \geq 2$, and write $e=xe_1+y[e_2,e_1] \in \mf[g]_w \cap \mf[g]_d = \mf[g]_{\alpha_1} \oplus \mf[g]_{\alpha_1+\alpha_2}$. Notice that $s_2(\alpha_1)=\alpha_1+a\alpha_2$ and $s_2(\alpha_1+\alpha_2)=\alpha_1+(a-1)\alpha_2$. Then by \cite{Ku2}*{Lemma 1.3.5}, $s_2$ gives rise to an automorphism $s_2(\ad) \in \Aut(\mf[g])$ satisfying 
$$
s_2(\ad): \mf[g]_{\alpha_1} \oplus \mf[g]_{\alpha_1+\alpha_2} \to \mf[g]_{\alpha_1+a\alpha_2} \oplus \mf[g]_{\alpha_1+(a-1)\alpha_2}.
$$
Further, as $\alpha_2(\dot{\tau})=0$, the action of $s_2(\ad)$ preserves the one-parameter subgroup $\tau$. Thus, we can consider $e':=s_2(\ad)(e) \in \mf[g]_{s_2 w} \cap \mf[g]_d$, which can be extended into an $\mf[sl]_2$-triple by the argument as above. Applying $s_2(\ad)^{-1}$ gives an $\mf[sl]_2$-triple for $e$, as desired.  
\end{proof}

All together, this completes the proof of the following theorem, which finishes the story for hyperbolic rank two Kac--Moody algebras. 

\begin{theorem} \label{rank 2 theorem} Let $\mf[g]=\mathcal{H}(a,b)$ for any $a, b \in \Z_{\geq 1}$ with $ab \geq 5$. Then $e \in \mathcal{H}(a,b)_w \cap \mathcal{H}(a,b)_d$ can be extended into an $\mf[sl]_2$-triple. 
\end{theorem}

\section{$\pi$-systems of real roots and regular gradings} \label{S4}

\subsection{$\pi$-systems and regular subalgebras} In this section, we recall the notion of a $\pi$-system of roots as in Carbone et. al. \cite{CRRV}. The study of $\pi$-systems in semisimple Lie algebras is due to Borel--de Siebenthal \cite{BdS} and Dynkin \cite{Dyn} under the more familiar name of \textit{regular semisimple subalgebras}. In the general symmetrizable setting, $\pi$-systems first appear in the works of Morita \cite{Mor} and Naito \cite{Nai}. We follow the conventions and notations as in \cite{CRRV}. We first begin with the definition.

\begin{definition} Let $\Phi$ be the root system of a symmetrizable Kac--Moody algebra $\mf[g]$. A \emph{$\pi$-system} is a finite set $\Sigma \subset \Phi_{Re}$ of distinct real roots such that $\beta_i-\beta_j \not \in \Phi$ for any $\beta_i, \beta_j \in \Sigma, \ i \neq j$. 
\end{definition}

For a $\pi$-system $\Sigma$, fix an ordering $\Sigma=\{ \beta_1, \beta_2, \dots, \beta_m\}$, and consider the $m \times m$ matrix 
$$
B:=(B)_{kj}=\beta_j(\beta^\vee_k).
$$
By root system considerations, we have that $B_{kk}=2$ for all $k=1, \dots, m$ and $B_{jk} \leq 0$ (as $\beta_k-\beta_j \not \in \Phi$), with $B_{jk}=0$ if and only if $B_{kj}=0$. That is, the matrix $B$ is a generalized Cartan matrix. We will say that the $\pi$-system $\Sigma$ is of \emph{type} $B$. Denote by $\mf[g](B)$ the Kac--Moody algebra with Cartan matrix $B$ and by $Q(B)$ its associated root lattice. Then if $Q$ is the root lattice for $\mf[g]$, we have a natural $\C$-linear, form-preserving map 
$$
q_\Sigma: Q(B) \otimes_{\Z} \C \to Q \otimes_{\Z} \C
$$
sending $\alpha_i \mapsto \beta_i$. The map $q_\Sigma$ can in fact be extended to the level of a Lie algebra homomorphism, as in the following results of \cite{CRRV} (cf. Theorem 2.5 and Proposition 2.6).

\begin{theorem} \label{pi system map} 
Let $\Sigma$ be a $\pi$-system in the root system $\Phi$ of $\mf[g]$. For each $\beta_i \in \Sigma$, choose nonzero vectors $e_{\beta_i} \in \mf[g]_{\beta_i}$ and $e_{-\beta_i} \in \mf[g]_{-\beta_i}$ such that $[e_{\beta_i}, e_{-\beta_i}]=\beta_i^\vee$. Then there exists a unique Lie algebra homomorphism 
$$
i_\Sigma: \mf[g]'(B) \to \mf[g]',
$$
where $\mf[g]':=[\mf[g], \mf[g]]$, such that $e_i \mapsto e_{\beta_i}$, $f_i \mapsto e_{-\beta_i}$, and $\alpha_i^\vee \mapsto \beta_i^\vee$. Furthermore, if $\Sigma$ is linearly-independent in $Q(A) \otimes_{\Z} \C$, one can extend the map $i_\Sigma$ to an embedding $i_\Sigma: \mf[g](B) \hookrightarrow \mf[g]$.
\end{theorem}

Throughout, we will always use linearly-independent $\pi$-systems, so that we are ensured by Theorem \ref{pi system map} a Lie algebra embedding. Further, note that since $q_\Sigma$ is form-preserving, we have if $\Phi(B)$ is the root system of $\mf[g](B)$ for a $\pi$-system of type $B$ that $q_\Sigma(\Phi(B)_{Re}) \subset \Phi_{Re}$ and $q_\Sigma(\Phi(B)_{Im}) \subset \Phi_{Im} \sqcup \{0\}$ (cf. \cite{CRRV}, Corollary 2.12). 

In the following proposition, we make a first connection between linearly-independent $\pi$-systems of type $B$, where $B$ is a \textit{non-degenerate} Cartan matrix, and $\mf[sl]_2$-triples. The construction is identical to that of the so-called \textit{principle $\mf[sl]_2$-triple} for semisimple Lie algebras $\mf[g]$ as in \cite{Kos}.

\begin{proposition} \label{pi reg} Let $\Sigma$ be a linearly independent $\pi$-system of type $B$, where $\det(B) \neq 0$. Let $e=\sum_{\beta \in \Sigma} e_\beta$ where each $0 \neq e_\beta \in \mf[g]_\beta$. Then $e$ can be extended into an $\mf[sl]_2$-triple. 
\end{proposition}

\begin{proof}
By definition, fix an ordering $\Sigma=\{\beta_i\}_{i=1}^{|\Sigma|}$ and $B$ the associated Cartan matrix as before. Since $\Sigma$ is linearly independent and non-degenerate, define $h \in \mf[h]$ via 
$$
h:= \sum_{\beta_i \in \Sigma} \mu_i \beta_i^\vee,
$$
where the coefficients $\mu_i$ are chosen such that $\beta_j(h)=2$ for all $j$; this is possible, as $B$ is invertible and this corresponds to the linear system $B^T \cdot \vec{\mu} = \begin{pmatrix} 2  \\ \vdots \\ 2 \end{pmatrix}$. Now, define $f \in \mf[g]$ by 
$$
f:= \sum_{\beta_i \in \Sigma} \mu_i e_{-\beta_i}
$$
where as before $e_{-\beta_i} \in \mf[g]_{-\beta_i}$ satisfies $[e_{\beta_i}, e_{-\beta_i}]=\beta_i^\vee$. Then by construction we have $\{e, h, f\}$ forms an $\mf[sl]_2$-triple.
\end{proof}

\begin{remark} The non-degeneracy of $B$ in the above proposition is crucial. For example, if $\mf[g]=A_n^{(1)}$, untwisted affine $\mf[sl]_{n+1}$, and we take the linearly-independent $\pi$-system $\Sigma=\{\alpha_0, \alpha_1, \dots, \alpha_n\}$, then the element $e:=e_0+e_1+\cdots+e_n$ \textbf{cannot} be made into an $\mf[sl]_2$-triple, but instead can be extended into a three-dimensional Heisenberg subalgebra  $\langle e, f, K \rangle$ where $f=f_0+f_1+\cdots +f_n$ and $K \in \mf[g]$ is the canonical central element of $\mf[g]$. 
\end{remark}

\subsection{Regular-dominant one-parameter subgroups and $\mf[sl]_2$-triples} 
Now, we employ this machinery in the context of our chosen elements $e \in \mf[g]_w \cap \mf[g]_d$, with an additional restriction. For the remainder of this section, we fix as before a Weyl group element $w \in W$ and one-parameter subgroup $\tau \subset T \subset G^{min}$; now, we include the stronger condition that $\tau$ is \textit{regular-dominant}, so that $\alpha_i(\dot{\tau}) >0$ for all $i$. In this case, we will say that the grading on $\mf[g]$ coming from $\tau$ is a \emph{regular grading}. Fix an integer $d \geq 1$, and consider the subset of real positive roots $\Phi_w^d:=\{\beta \in \Phi_w: \beta(\dot{\tau})=d\}$. A first observation which follows from our strengthened assumptions is the following lemma.

\begin{lemma} \label{pi system lemma} For a regular grading, $\Phi_w^d$ is a $\pi$-system. 
\end{lemma}

\begin{proof} Suppose we have $\alpha, \beta \in \Phi_w^d$ such that $\alpha-\beta \in \Phi$. Without loss of generality, suppose $\alpha-\beta$ is a positive root. Then $(\alpha-\beta)(\dot{\tau}) >0$ by the regularity of $\tau$. But, $\alpha(\dot{\tau})=\beta(\dot{\tau})=d$, so that $\alpha(\dot{\tau})-\beta(\dot{\tau})=0$, a contradiction. Thus, no pairwise differences of roots in $\Phi_w^d$ can be a root, hence $\Phi_w^d$ is a $\pi$-system ($\Phi_w^d \subset \Phi^+_{Re}$ by construction). 
\end{proof}

By Proposition \ref{pi reg}, if we could show that $\Phi_w^d$ is linearly-independent, we can extend the associated ``principle nilpotent"-like element into an $\mf[sl]_2$-triple. However, we in fact have the stronger result for such $\Phi_w^d$, given by the following proposition.

\begin{proposition} \label{reg grade}
For a regular grading, the $\pi$-system $\Phi_w^d$ is of finite type. That is, $\mf[g](B)$ is a finite-dimensional semisimple Lie algebra. 
\end{proposition}

\begin{proof}
To show that $\Phi(B)$ is of finite type, we use the classification as in \cite{Kac}*{Chapter 5} that a root system is of finite type if and only if it has no imaginary roots.

First, we claim that $\Z_{\geq 0}. \Phi_w \cap \Phi \subset \Phi_w$. Indeed, if we have $\gamma \in \Z_{\geq 0} \Phi_w \cap \Phi$, necessarily $\gamma \in \Phi^+$. Writing $\gamma = \sum_{j} n_j \gamma_j$ for $\gamma_j \in \Phi_w$, since $w^{-1}(\gamma_j) \in \Phi^-$ and $n_j \geq 0$ for all $j$, we have that $w^{-1}(\gamma) \in \Phi^-$, so that by definition $\gamma \in \Phi_w$. 

Now, let $\Sigma=\Phi_w^d$ and suppose that $\gamma:= \sum_i n_i \alpha_i \in \Phi(B)$ is an imaginary root; without loss of generality, we take $\gamma$ to be positive. By \cite{CRRV} Corollary 2.12 (as in the discussion following Theorem \ref{pi system map} above), we have that 
$$
q_\Sigma(\gamma)=\sum_i n_i \beta_i \in \Phi^+_{Im} \sqcup \{0\}.
$$
However, as this is a nonnegative combination of roots in $\Phi_w$, if $q_\Sigma(\gamma) \neq 0$, we have that $q_\Sigma(\gamma) \in \Phi_w \subset \Phi^+_{Re}$, a contradiction. Thus necessarily $q_\Sigma(\gamma)=0$. 

Finally, evaluating on $\dot{\tau}$, we see that 
$$
0=\sum_i n_i \beta_i(\dot{\tau}) = \sum_i n_i (d) = d\sum_i n_i \implies \sum_i n_i =0,
$$
as $d \geq 1$. But since each $n_i \geq 0$, we get that $n_i=0$ for all $i$, so that $\gamma=0$. Thus $\Phi(B)$ has no imaginary roots, so is of finite type. 
\end{proof}

\begin{remark} Since $\Phi^d_w$ is a $\pi$-system of finite type with Cartan matrix $B$, we know that $\det(B) \neq 0$ so that $\Phi_w^d$ is linearly-independent and gives rise to an embedding $\mf[g](B) \hookrightarrow \mf[g]$. 
\end{remark}

With this, we can now state the main result of this section. 

\begin{theorem} \label{regdomthm} Let $\mf[g]$ be a symmetrizable Kac--Moody algebra, $w \in W$ a Weyl group element, and $\tau \subset T \subset G^{min}$ a regular-dominant one-parameter subgroup. Fix $d \geq 1$, and a nonzero element $e \in \mf[g]_w \cap \mf[g]_d$. Then $e$ can be extended into an $\mf[sl]_2$-triple. 
\end{theorem}

\begin{proof}
For such an $e$, we can write $e=\sum_{\beta_i \in \Phi_w^d} e_{\beta_i}$ for some (not necessarily all nonzero) $e_{\beta_i} \in \mf[g]_{\beta_i}$. By Lemma \ref{pi system lemma} and Proposition \ref{reg grade}, we have that $\Phi_w^d$ is a $\pi$-system of finite type corresponding to an embedding $i_\Sigma: \mf[g](B) \hookrightarrow \mf[g]$, with $e \in i_\Sigma(\mf[g](B))$. Note that the preimage $i_\Sigma^{-1}(e)$ is nilpotent in $\mf[g](B)$, as $e$ is locally nilpotent in $\mf[g]$. As this is a Lie algebra embedding and $\mf[g](B)$ is a finite-dimensional semisimple Lie algebra, we can apply the classical Jacobson--Morozov theorem to $i^{-1}_\Sigma(e)$ to construct an $\mf[sl]_2$-triple, which is then transported via $i_\Sigma$ to an $\mf[sl]_2$-triple for $e$.

\end{proof}

\begin{remark} For both Theorems \ref{rank 2 theorem} and \ref{regdomthm}, the construction of an $\mf[sl]_2$-triple for $e \in \mf[g]_w \cap \mf[g]_d$ followed from the classical Jacobson--Morozov theorem, by appropriately recognizing the element $e$ in a natural finite-dimensional semisimple subalgebra. We do not know if this behavior should continue in the general case. 
\end{remark}

\begin{remark} Question \ref{Q1}, as stated for a general symmetrizable $\mf[g]$ and dominant one-parameter subgroup $\tau$, remains open. In an appendix to our Ph.D. thesis \cite{Jer}, we record an argument of Kumar \cite{Ku} (adapting the proof in the finite-dimensional semisimple case as given by Kostant \cite{Kos}) which answers Question \ref{Q1} up to an (unresolved) sufficient condition. We do not include that argument here, for brevity. 
\end{remark}

\end{document}